\documentclass[francais]{smfart}

\usepackage[english,frenchb]{babel}
\usepackage[latin1]{inputenc}
\usepackage[T1]{fontenc}
\usepackage{smfthm}
\usepackage{lmodern}

\newcommand{\dr}{\mathbb}

\renewcommand\d{\textnormal}

\newcommand\bG{{\bf G}}
\newcommand\G{\Gamma}
\newcommand\g{\gamma}

\newcommand\La{\Lambda}
\newcommand\la{\lambda}
\newcommand\GG{{G / \G}}
\newcommand\GH{{G^f  \backslash G}}

\newenvironment{fait}{\begin{enonce}{Fait}}{\end{enonce}}

\title[\'Equidistribution des matrices de d\'enominateur $n$]{Existence et
\'equidistribution des matrices de d\'enominateur
$n$ dans les groupes unitaires et orthogonaux}
\author{Antonin Guilloux}
\address{D\'epartement de Math\'ematiques et
Applications\\ \'Ecole Normale Sup\'erieure\\45 rue d'Ulm\\75005 Paris\\
France.}
\email{antonin.guilloux@ens.fr}
\urladdr{http://www.dma.ens.fr/{\lower.7ex\hbox{\~{}}}aguillou/}
\date{4 septembre 2007}

\begin{document}

\begin{altabstract}
Let $\bG$ be a simply-connected $\dr Q$ quasisimple and $\dr
R$ anisotropic algebraic
$\dr Q$-group. Let $\dr A^f$ be the finite part of the ad\`eles $\dr A$ of $\dr
Q$. Let $(H_n)$ be a sequence of bounded subsets of
${\bf G}(\dr A^f)$ which are bi-invariant by a compact open subgroup of ${\bf
G}(\dr A^f)$. Let $\G_n$ be the projection in ${\bf G}(\dr R)$ of the sets ${\bf
G}(\dr Q) \cap ({\bf G}(\dr R)\times H_n)$. Suppose that the volume of the
compact subsets ${\bf G}(\dr R)\times H_n$ tends to $\infty$ with $n$. We prove
the equidistribution in ${\bf G}(\dr R)$ of the $\G_n$ with respect to
the Haar probability on ${\bf G}(\dr R)$. The strategy is to use a mixing
result for the action of $\bG(\dr A)$ on the space $L^2(\bG(\dr A)/\bG(\dr Q))$.
As an application, we study the existence and the repartition of
rational unitary matrices having a given denominator. We prove a local-global
principle for this problem and the equirepartition of the sets of denominator
$n$-matrices when they are not empty. Then we study the more complicated case of
non simply-connected groups applying it to quadratic forms.
\end{altabstract}
\maketitle

\section{Introduction}

La th\'eorie des formes quadratiques d\'efinies positives \`a coefficients
entiers r\'epond de
mani\`ere satisfaisante aux deux questions suivantes :
\begin{itemize}
 \item \`A quelles conditions une forme quadratique donn\'ee repr\'esente un
entier $n$ (c'est-\`a-dire qu'il existe un vecteur entier de norme $\sqrt{n}$) ?
 \item Quand un entier $n$ est repr\'esent\'e, quelle est la r\'epartition des
vecteurs entiers sur l'ellipso\"ide des vecteurs de norme $\sqrt{n}$ ?
\end{itemize}
Citons les r\'esultats les plus simples, qui sont obtenus quand le rang de la
forme quadratique est au moins $5$ :

\begin{theo}[W. Tartakowsky \cite{tartakowsky}, C. Pommerenke
\cite{pommerenke}]
Soit $q$ une forme quadratique d\'efinie positive de rang $k \geq 5$ \`a
coefficients
entiers. Alors il existe un entier $N_0$ tel que pour tout $n\geq N_0$, on a
l'\'equivalence entre les deux assertions suivantes :
\begin{enumerate}
\item Pour tout nombre premier $p$, l'entier $n$ appartient \`a $q(\dr Z_p^k)$.
\item L'entier $n$ appartient \`a $q(\dr Z^k)$.
\end{enumerate}

De plus, l'ensemble des vecteurs $v$ de $\dr Z^k$ v\'erifiant $q(v)=n$
s'\'equir\'epartit sur l'ellipso\"ide $q(x)=n$ quand $n$ tend vers l'infini.
\end{theo}

Nous reviendrons dans la partie \ref{orth} sur ce th\'eor\`eme et sur le cas des
formes de petit rang. Nous nous int\'eressons dans ce texte \`a un analogue
dans le cadre des groupes unitaires ou orthogonaux de ce r\'esultat.

\subsection{Matrices de d\'enominateur $n$ dans le groupe unitaire}

Pr\'esentons maintenant nos r\'esultats
dans le cas unitaire (pour le cas orthogonal, on renvoie \`a nouveau \`a la
partie \ref{orth}). Pour tout entier $k\geq 2$ et tout anneau $A$, on note
$\d{M}(k,A)$ l'ensemble des matrices carr\'ees de taille $k\times k$ \`a
coefficients dans $A$.
D\'efinissons le d\'enominateur d'une matrice \`a coefficients dans $\dr Q[i]$
:

\begin{defi}
Soient $k$ un entier et $A$ une matrice de $\d{M}(k,\dr Q[i])$.

Le d\'enominateur $d$ de $A$ est d\'efini comme le plus petit entier $d\in \dr
N^*$ tel que $d A$ soit une matrice de $\d{M}(k,\dr Z[i])$.
\end{defi}

On fixe $k \geq 2$. Soient $H\in \d{M}(k,\dr Z[i])$ une
matrice hermitienne d\'efinie positive, $h$ la forme hermitienne associ\'ee.
Nous voulons comprendre le comportement de l'ensemble des matrices de
$\d{SU}(h,\dr Q)$ de d\'enominateur $n$ : \`a quelles conditions cet ensemble
est
non vide, et dans ce cas, quelle est sa r\'epartition dans le groupe
$\d{SU}(h,\dr R)$. On peut reformuler le probl\`eme de la fa\c con suivante :
 $A$ d\'esigne un des anneaux $\dr Z$ ou $\dr Z_p$ (pour $p$ premier), et
$A_i=\dr
Z[i] \otimes_{\dr Z} A$.
Pour tout entier $n$, on note $\mathcal T(n,H,A)$ l'ensemble des
matrices $M\in \d{M}(k,A_i)$ de d\'eterminant $n^k$ telles que :
\begin{itemize}
\item les coefficients de $M$ sont premiers entre eux,
\item la matrice $M$ est solution de l'\'equation $(E_n)$ : $M^*HM= n^2 H$.
\end{itemize}

Dans le cas $A=\dr Z$ et pour tout entier $n$, une matrice $M$
est dans $\mathcal T(n,H,\dr Z)$ si et seulement si la matrice $\frac{1}{n}M$
est un \'el\'ement de $\d{SU}(h,\dr Q)$ de d\'enominateur $n$. De m\^eme dans le
cas
$A=\dr Z_p$, une matrice $M$ est dans $\mathcal T(n,H,\dr Z_p)$ si et seulement
si la matrice $\frac{1}{n}M$ est un \'el\'ement de $\d{SU}(h,\dr Q_p)$ tel que
le
supremum de la norme $p$-adique des coefficients soit la norme $p$-adique de
$\frac{1}{n}$.

On note enfin $\mathcal U(H,A)$ l'ensemble
des $n \in \dr Z$ tels qu'il existe $M \in \mathcal T(n,H,A)$, et
$\displaystyle \mathcal U_l(H)=\bigcap_{p \textrm{ premier}} \mathcal U(H,\dr
Z_p)$.
Bien s\^ur, pour qu'il existe des matrices de d\'enominateur $n$ dans
$\d{SU}(h,\dr
Q)$, il faut que $n$ soit dans $\mathcal U(H,\dr Z)$, et donc il faut que $n$
soit dans $\mathcal U_l(H)$.

Le th\'eor\`eme suivant assure que pour $n$ suffisamment grand, c'est la seule
condition :

\begin{theo}\label{coro:existenceuni}
Soient $k \geq 2$, $H\in \d{M}(k,\dr Z[i])$ une
matrice hermitienne d\'efinie positive. Alors il existe $N_0 \in \dr N$ tel
que pour tout $n\geq N_0$, les deux assertions suivantes sont \'equivalentes :
\begin{enumerate}
\item L'entier $n$ appartient \`a $\mathcal U_l(H)$.
\item L'entier $n$ appartient \`a $\mathcal U(H,\dr Z)$.
\end{enumerate}
\end{theo}

\subsection{\'Equir\'epartition des matrices rationnelles}

La m\'ethode pour prouver ce th\'eor\`eme est de prouver un r\'esultat plus
fort, \`a savoir l'\'equir\'epartition dans $\d{SU}(h,\dr R)$ de l'ensemble
$\G_n$ des matrices de d\'enominateur $n$, quand $n$ tend vers l'infini dans
$\mathcal U_l(H)$. Voici l'\'enonc\'e :

\begin{theo}\label{the:equiuni}
Soient $k \geq 2$, $H\in \d{M}(k,\dr Z[i])$ une
matrice hermitienne d\'efinie positive et $h$ la forme hermitienne
associ\'ee. Notons $\mu$ la probabilit\'e de Haar sur $\d{SU}(h,\dr R)$. Pour
tout entier $n$, soit $\G_n$ l'ensemble des $ \frac{1}{n}M \textrm{ pour
} M\in \mathcal T(n,H,\dr Z)$.

Alors quand $n$ tend vers l'infini dans $\mathcal U_l(H)$, les $\G_n$
s'\'equir\'epartissent dans $\d{SU}(h,\dr R)$, c'est-\`a-dire qu'on la
convergence :

$$\frac{1}{\d{Card} (\G_n)} \sum_{\g \in \G_n} \delta_{\g} \xrightarrow[n
\in \mathcal U_l(H)]{n \to \infty} \mu$$
dans l'espace des probabilit\'es sur $\d{SU}(h,\dr R)$ muni de la topologie
faible-$\ast$.
\end{theo}

\begin{rema}
Pour v\'erifier la condition $n \in \mathcal U_l(H)$, il suffit de v\'erifier
que l'ensemble $\mathcal T(n,H,\dr Z_p)$ est non vide pour les
nombres premiers $p$ divisant $n$. En effet, si $p$ ne divise pas $n$, la
matrice $n\cdot Id$ appartient \`a $\mathcal T(n,H,\dr Z_p)$.
\end{rema}

La question de la r\'epartition des points rationnels de
d\'enominateur $n$ dans le
groupe des points r\'eels d'un groupe alg\'ebrique $\bG$ d\'efini sur $\dr Q$
quand $n$ tend vers l'infini a d\'ej\`a \'et\'e \'etudi\'ee par plusieurs
auteurs.

Dans le cas o\`u $\bG(\dr R)$ est non-compact, A. Eskin et H. Oh
\cite{eskin-oh}
ont d\'emontr\'e que ces points \'etaient \'equidistribu\'es suivant la mesure
de
Haar de $\bG(\dr R)$. Pour cela ils utilisent la pr\'esence de
sous-groupes unipotents dans $\bG (\dr R)$ et concluent gr\^ace \`a des
th\'eor\`emes de Ratner et Dani-Margulis. Cependant, dans le cas o\`u $\bG (\dr
R)$ est compact, il n'y a pas d'\'el\'ement unipotent dans $\bG (\dr R)$, donc
on ne peut
pas appliquer ces th\'eor\`emes.

Une autre m\'ethode pour prouver des th\'eor\`emes
d'\'equir\'epartition est d'utiliser le m\'elange. On renvoie \`a l'article
d'A. Eskin et C. McMullen \cite{Eskin-MacMullen} pour une pr\'esentation
tr\`es
claire de cette m\'ethode. Pour pouvoir l'utiliser dans notre cas,
il faut disposer d'un r\'esultat de d\'ecroissance des coefficients de l'action
de $\bG$ sur $L^2 \left( \bG(\dr A) / \bG(\dr Q) \right)$ (dans
ce cadre $\bG(\dr Q)$ est un r\'eseau du groupe des points sur les ad\`eles
$\bG(\dr A)$). De tels r\'esultats
sont prouv\'es dans l'article de L. Clozel, H. Oh et E. Ullmo
\cite{clozel-oh-ullmo},
 et compl\'et\'es dans un article de L. Clozel \cite{clozel}, puis de
A. Gorodnik,
F. Maucourant et H. Oh \cite{gorodnik-oh-maucourant} o\`u une d\'ecroissance
des coefficients de l'action de $\bG$ sur $L^2 \left( \bG(\dr A) / \bG(\dr Q)
\right)$ est prouv\'ee sous les hypoth\`eses que $\bG$ est un groupe
alg\'ebrique d\'efini sur un corps de nombres, connexe et absolument
presque-simple (on renvoie au th\'eor\`eme \ref{the:dec} pour l'\'enonc\'e
exact).

C'est ce dernier r\'esultat que nous utiliserons. Commen\c cons par rappeler
quelques r\'esultats sur les
groupes alg\'ebriques et leurs r\'eseaux arithm\'etiques, ce qui permettra de
fixer les notations.

\subsection{Notations}\label{notations}

Nous d\'efinissons dans cette partie les notations dont nous nous servirons
dans ce texte. Il nous faudra pour cela faire appel \`a des r\'esultats sur
les groupes alg\'ebriques et ad\'eliques. Pour leur preuve, nous
renvoyons le lecteur d'une part  \`a l'article \cite{Tits} de J. Tits (et
ses r\'ef\'erences) pour les r\'esultats sp\'ecifiques aux points sur $\dr
Q_p$ d'un groupe alg\'ebrique, et d'autre part au livre de V. Platonov et
A. Rapinchuk \cite{platonov-rapinchuk} pour les propri\'et\'es ad\'eliques.

Fixons une fois pour toutes un
groupe ${\bf G}$ d\'efini sur un corps de nombres $K$, connexe,
presque-$K$-simple (c'est-\`a-dire que tout sous-$K$-groupe distingu\'e de $\bG$
est fini). Soit $\mathcal V$ l'ensemble des
places de $K$. On note $\dr A$ (resp. $\dr A^f$, resp. $\dr A^{\infty}$)
l'anneau des ad\`eles de $K$
(resp. des ad\`eles finies, resp. infinies). On note de plus
$$G={\bf G} (\dr A)\d{ ; }G^f={\bf G}(\dr A^f)\textrm{ et }G^{\infty}={\bf G}
(\dr
A^{\infty}).$$ On supposera toujours que $G^\infty$ est compact.

Nous disposons alors du sous-groupe  ${\bf
G}(K)$. On rappelle que c'est un r\'eseau de $G$, qu'il est
irr\'eductible car ${\bf G}$ est presque-$K$-simple ; et enfin
qu'il est cocompact car ${\bf G}$ est $K$-anisotrope.

On appelle r\'eseau arithm\'etique
de $G$ tout sous-groupe $\G$ tel que $\G \cap{\bf G}(K)$ est d'indice fini
dans $\G$ et dans ${\bf G}(K)$.
D'apr\`es ce qui pr\'ec\`ede, tout r\'eseau arithm\'etique $\G$ de $G$ est
irr\'eductible et cocompact. On se fixe un tel r\'eseau $\G$,
ainsi qu'un sous-groupe compact ouvert $U$ de $G^f$.

On dira qu'une suite d'\'el\'ements $(g_n)$ de $G$ tend vers l'infini si pour
toute partie compacte $C$ de $G$, pour $n$ suffisamment grand, $g_n$
n'appartient pas \`a
$C$.

On note $\tau^{\infty}$ la projection de $G$
sur $G^{\infty}$,
$\tau^f$ la projection de $G$ sur $G^f$, et enfin $\pi$ la projection de
$G$ sur $\GG$. De plus on note $\lambda$ la mesure de Haar sur $G^f$
normalis\'ee par $\lambda(U)=1$ ; $\mu$ la probabilit\'e de Haar sur
$G^{\infty}$. On note enfin $m$ la probabilit\'e sur $\GG$ localement
proportionnelle \`a
$\mu\otimes \lambda$ et on l'appelle
probabilit\'e de Haar sur $\GG$. Le
diagramme ci-dessous r\'esume ces donn\'ees :

$$
\begin{matrix}
& & G \textrm{ , } \mu \otimes \lambda & &\\
& & & & \\
  & \tau^{\infty} \swarrow   & \tau^f \downarrow &  \searrow \pi &     \\
& & & &\\
G^{\infty}=\GH \textrm{ , }\mu & & G^f \textrm{ , } \la  & & \GG \textrm{ , }
m\\

\end{matrix}
$$

Nous utilisons les notations suivantes pour les fonctions caract\'eristiques
et les masses de Dirac dans un ensemble $A$ : si $a$ est un \'el\'ement de $A$,
on note $\delta_a$ la masse de Dirac en $a$ ; si $B$ est une partie de $A$, on
note ${\bf 1}_B$ la fonction caract\'eristique de $B$.

Enfin, pour les applications, on supposera toujours fix\'ee une base sur $\dr
Z^k$ et $\dr Z[i]^k$. Ainsi, nous supposons fix\'ee une fois pour toutes
l'identification entre formes quadratiques (resp. hermitiennes) et matrices
sym\'etriques (resp. hermitiennes).

\subsection{Application de la d\'ecroissance des coefficients}

On remarque qu'une matrice est de d\'enominateur $n$ si et seulement si pour
tout nombre premier $p$, le maximum de la norme $p$-adique de ses coefficients
est la norme $p$-adique de $\frac{1}{n}$. Donc on peut r\'e\'enoncer
notre probl\`eme comme un probl\`eme de r\'epartition dans $G^\infty$ de
sous-ensembles de $\G$ d\'efinis par certaines conditions sur leur projection
dans $G^f$. C'est dor\'enavant sous cet angle que nous travaillerons.

Nous prouvons dans ce cadre le r\'esultat d'\'equir\'epartition suivant
(rappelons que $U$ est un sous-groupe compact ouvert fix\'e de $G^f$) : pour une
suite
d'ensembles $H_n \subset G^f$ bi-$U$-invariants, notons $\G_n$ l'ensemble des
points de $\G$ dont la projection dans $G^f$ appartient \`a $H_n$. Alors, si le
cardinal des $\G_n$ tend vers l'infini, ils s'\'equir\'epartissent dans
$G^\infty$ vers la mesure de Haar sur $G^\infty$.

C'est l'objet du th\'eor\`eme suivant (pour lequel on utilise les notations
d\'efinies en \ref{notations}) :

\begin{theo}\label{the:appl}
Soit ${\bf G}$ un $K$-groupe, presque-$K$-simple, connexe tel que le
compl\'et\'e aux places archim\'ediennes $G^{\infty}$ est compact.
Soient $U$ un sous-groupe compact ouvert du groupe des ad\`eles finies $G^f$ et
$(H_n)$ une suite de sous-ensembles compacts bi-$U$-invariants de $G^f$.
Soit $\G$ un sous-groupe arithm\'etique du groupe des ad\`eles $G$ et $\G_n=\G
\cap (G^{\infty}\times H_n)$.
Supposons que $\d{Card}(\G_n)$ tende vers l'infini.

Alors, la projection de $\G_n$ dans $G^\infty$ s'\'equir\'epartit dans
$G^\infty$ ; c'est-\`a-dire qu'on a la convergence :
$$\lim_{n\to \infty} \frac{1}{\d{Card} (\G_n)} \sum_{\g \in
\G_n} \delta_{\tau^{\infty}(\g)} = \mu \;.$$
dans l'espace des
probabilit\'es sur $G^{\infty}$.
\end{theo}

Nous obtiendrons avec le th\'eor\`eme \ref{the:appl2} un \'equivalent
de $\d{Card}(\G_n)$. Ce th\'eor\`eme sera
prouv\'e dans la partie \ref{thm}. De plus, on notera toujours
$G_n=G^\infty\times H_n$.

\bigskip

\emph{Remerciements :} L'auteur tient \`a remercier Y. Benoist pour les
nombreuses discussions et les nombreux conseils qu'il lui a prodigu\'es.
L'auteur remercie \'egalement le referee pour l'attention accord\'ee \`a
ce travail et les int\'eressantes remarques formul\'ees.

\section{D\'ecroissance des coefficients}

\subsection{Le th\'eor\`eme de d\'ecroissance}

Nous pr\'esentons dans cette partie le th\'eor\`eme de A. Gorodnik, F.
Maucourant et H. Oh. Pour cela, il nous
faut comprendre la repr\'esentation de $G$ dans l'espace $L^2(\GG)$.

On note $\langle ,\rangle $
le produit scalaire canonique dans $L^2(\GG)$ et $g.f$ l'action de $g \in G$ sur
une fonction $f$ de $L^2(\GG)$, donn\'ee par $(g.f)(x\G) =f(g^{-1}x\G)$.
Consid\'erons $\La$ l'ensemble des
caract\`eres unitaires de $G$ triviaux sur $\G$. Ils forment une base du
sous-espace vectoriel de $L^2(\GG)$ engendr\'e par les sous-repr\'esentations
de dimension 1 de $G$. Nous notons $L^2_0(\GG)$ l'orthogonal de ce sous-espace
dans $L^2(\GG)$.

\begin{rema}
Dans \cite{gorodnik-oh-maucourant}, la repr\'esentation de $G$ consid\'er\'ee
est la repr\'esentation $\la$ de $G$ dans $L^2(\G \backslash G)$ donn\'ee par
$(\lambda(g)\varphi)(\G x)=\varphi(\G xg)$.

Consid\'erons l'isom\'etrie $\Psi$ entre $L^2(\GG)$ et $L^2(\G \backslash G)$
donn\'ee par l'\'egalit\'e : $\Psi(f)(\G x)= f(x^{-1} \G)$. Alors, on a pour
tout $f$ dans
$L^2(\GG)$ et $g$ dans $G$, $\la(g)\Psi(f)=\Psi(g.f)$. Cette \'egalit\'e nous
permet d'utiliser les r\'esultats de \cite{gorodnik-oh-maucourant} dans notre
cas.
\end{rema}

Nous pouvons maintenant citer le th\'eor\`eme de A. Gorodnik, F. Maucourant et
H. Oh \cite[Corollaire 1.20]{gorodnik-oh-maucourant} :1

\begin{theo}\label{the:dec}
Soit $\bG$ un groupe d\'efini sur un corps de nombres $K$, connexe et
absolument presque-simple. Alors pour toutes fonctions $f$ et $h$ de
$L^2_0(\GG)$, on a :

$$|\langle f,g.h\rangle | \xrightarrow{g \to \infty} 0$$
\end{theo}

Remarquons que dans \cite{gorodnik-oh-maucourant}, le produit scalaire
est major\'e gr\^ace \`a une fonction $\bar \xi$ construite de mani\`ere
explicite. Nous n'utiliserons pas ici cette estim\'ee. Comme ceci permet de
simplifier la preuve, nous en donnons un aper\c cu dans la partie suivante.

\subsection{Cas des groupes unitaires et orthogonaux}

Nous voulons dans cette partie donner une id\'ee de la preuve du th\'eor\`eme
pr\'ec\'edent dans un cas simple mais pertinent pour nos applications : $\bG$
est le groupe $\d{\d{SO}}(q)$ o\`u $q$
est une forme quadratique d\'efinie positive rationelle en au moins 5 variables
(la m\'ethode qu'on pr\'esente fonctionne sans modification pour un groupe
$\d{SU}(h)$, $h$ de rang au moins 4). Pr\'ecis\'ement, nous prouvons la
proposition suivante :

\begin{prop}\label{caspart}
Soit $k \geq 5$ et soit $q$ une $\dr Q$-forme quadratique d\'efinie positive sur
$\dr
Q^k$. Pour toutes fonctions $f$ et $g$ de $L^2_0(\d{SO}(q,\dr A)/\d{SO}(q,\dr
Q))$, on a :
$$|\langle f,g.h\rangle | \xrightarrow[g\in \d{SO}(q,\dr A)]{g \to \infty} 0$$
\end{prop}

La preuve que nous donnons reprend les id\'ees de
\cite{gorodnik-oh-maucourant}. Comme nous ne cherchons qu'une d\'ecroissance des
coefficients et non pas une majoration, nous pouvons par endroits aller un peu
plus vite, par exemple en utilisant le th\'eor\`eme de
Howe-Moore \cite[Th\'eor\`eme 10.1.4]{zimmer}.

\begin{proof}
Nous devons d'abord nous assurer d'un fait qui nous permettra d'utiliser les
r\'esultats de H. Oh \cite{oh} et le th\'eor\`eme de Howe-Moore :
\begin{fait}
Soit $p$ un nombre premier tel que le groupe $\d{SO}(q,\dr Q_p)$ est
non-compact. Soit
$\d{SO}(q,\dr Q_p)^+$ le sous-groupe de $\d{SO}(q,\dr Q_p)$ engendr\'e par les
\'el\'ements unipotents.

Alors il n'existe pas de vecteurs $\d{SO}(q,\dr Q_p)^+$-invariants dans
l'espace
$L^2_0(\d{SO}(q,\dr A)/\d{SO}(q,\dr Q))$.
\end{fait}

Ce fait est un cons\'equence de l'approximation forte \cite[Paragraphe
7.2]{platonov-rapinchuk}. Nous renvoyons pour sa d\'emonstration au lemme 3.8 de
\cite{gan-oh} et \`a sa preuve.

Soit $U$ le sous-groupe compact $\displaystyle \prod_{p \d{ premier}}
\d{SO}(q,\dr Z_p)$ de $\d{SO}(q,\dr A)$. Par d\'efinition, une
fonction $f$ dans l'espace $L^2_0(\d{SO}(q,\dr A)/\d{SO}(q,\dr Q))$ est
$U$-finie si l'espace vectoriel engendr\'e par $U.f$ est de dimension finie.
C'est une cons\'equence classique du th\'eor\`eme de Peter-Weyl que
l'ensemble des fonctions $U$-finies
est dense dans $L^2_0(\d{SO}(q,\dr A)/\d{SO}(q,\dr Q))$. Donc il
suffit de d\'emontrer la proposition \ref{caspart} pour de telles fonctions. Si
$f$ est $U$-finie, on note $d(f)$ la dimension de l'espace engendr\'e par
$U.f$. Pour $p$ un nombre premier, on d\'efinit de m\^eme la notion de fonction
$\d{SO}(q,\dr Z_p)$-finie et, si $f$ est un fonction $\d{SO}(q,\dr Z_p)$-finie,
on note $d(f,p)$ la dimension de l'espace engendr\'e par $(\d{SO}(q,\dr
Z_p).f)$. Bien s\^ur, si $f$ est $U$-finie, elle est $\d{SO}(q,\dr Z_p)$-finie.
Dans ce cas, on a $d(f)\geq d(f,p)$.

L'hypoth\`ese sur $k$ nous assure que, pour $p$ suffisamment grand, le groupe
$\d{SO}(q,\dr Q_p)$ est de rang au moins $2$ et qu'on dispose d'une
d\'ecomposition de Cartan $\d{SO}(q,\dr Q_p)=\d{SO}(q,\dr Z_p) A^+ \d{SO}(q,\dr
Z_p)$ o\`u $A^+$ est un sous-semigroupe d'un tore d\'eploy\'e maximal. En
effet ceci est une cons\'equence de r\'esultats de cohomologie galoisienne
\cite[Th\'eor\`eme 6.7]{platonov-rapinchuk} ; mais aussi de mani\`ere plus
\'el\'ementaire du lemme \ref{rang2} dont nous avons besoin dans la partie
\ref{orth}. Notons $\mathcal P_2$ l'ensemble des nombres premiers v\'erifiant
ces deux propri\'et\'es, et $\mathcal P_1$ le compl\'ementaire de $\mathcal P_2$
dans l'ensemble $\mathcal V$ des places de $\dr Q$.

Pour un groupe de rang $2$ sur $\dr Q_p$, une majoration
uniforme des coefficients de n'importe quelle repr\'esentation unitaire
est donn\'ee par les r\'esultats de H. Oh \cite{oh}. Nous r\'esumons ce que nous
utilisons de \cite{oh} dans le fait suivant, o\`u l'on consid\`ere la
repr\'esentation unitaire de
$\d{SO}(q,\dr Q_p) \subset \d{SO}(q,\dr A)$ dans $L^2_0(\d{SO}(q,\dr
A)/\d{\d{SO}}(q,\dr Q))$ :

\begin{fait}[Oh \cite{oh}]\label{Oh}
Il existe pour tout $p$ dans $\mathcal P_2$ une fonction $\xi_p \; : \;
\d{SO}(q,\dr Q_p) \to ]0,1]$
v\'erifiant les conditions suivantes :
\begin{itemize}
 \item pour tout  $g_p \in \d{SO}(q,\dr Q_p)$, on a
$\xi_p(g_p) \xrightarrow{g_p \to \infty }0$.
\item pour tout $g_p \in \d{SO}(q,\dr Q_p)$, si $g_p$ n'appartient pas \`a
$\d{SO}(q,\dr Z_p)$, on a $\xi_p(g_p) \leq \frac{2}{\sqrt{p}}$,
\item pour toutes fonctions $f$ et $h$ de
$L^2_0(\d{SO}(q,\dr A)/\d{SO}(q,\dr Q))$ qui sont $U$-finies,
on a : $$\langle f,g_p.h\rangle  \leq \sqrt{d(f,p) d(h,p)} \|f\|_2\|h\|_2
\xi_p(g_p) \;.$$
\end{itemize}
\end{fait}

Ce fait nous permet de g\'erer les nombres premiers de $\mathcal
P_2$. Pour g\'erer les autres, qui sont en nombre fini par hypoth\`ese, nous
pouvons faire appel au th\'eor\`eme de Howe-Moore \cite[Th\'eor\`eme
10.1.4]{zimmer} qui donne la d\'ecroissance des coefficients pour toute
repr\'esentation d'un produit fini de groupes $p$-adiques simples sans vecteurs
invariants par le sous-groupe engendr\'e par les \'el\'ements unipotents. Nous
\'enon\c
cons une version ad-hoc de ce th\'eor\`eme :

\begin{fait}[Howe-Moore \cite{zimmer}]
Soit $G_{\mathcal P_1}$ le groupe $\prod_{p \in \mathcal
P_1} \d{SO}(q,\dr Q_p)$. Alors pour tous sous-ensembles compacts $F$ et $H$ de
$L^2_0(\d{SO}(q,\dr A)/\d{SO}(q,\dr Q))$, on a la convergence :
$$\d{sup}_{f\in F \; ; \; h\in H}\;|\langle f,g.h\rangle | \xrightarrow[g\in
G_{\mathcal
P_1}]{g \to \infty} 0$$
\end{fait}

Fixons maintenant $f$ et $h$ deux fonctions $U$-finies, normalis\'ees pour que
$\|f\|_2=\|h\|_2=1$.
Soit $(g_n=(g_{p,n})_{p\in \mathcal V})$ une suite d'\'el\'ements de
$\d{SO}(q,\dr A)$ qui tend
vers l'infini. Pour tout $n$, on note : $$g^1_n=(g_{p,n})_{p \in\mathcal P_1}
\d{ et } g^2_n=(g_{p,n})_{p \in\mathcal P_2}\;.$$ De plus, pour nombre premier
$p$, on
note
$\overline {g_{p,n}}=(g_{q,n})_{q \neq p}$. Quitte \`a partitionner la suite
$(g_n)$ en
sous-suites, on a l'alternative suivante :
\begin{enumerate}
 \item La suite $(g^2_n)$ reste dans une partie compacte.
\item La suite $(g^2_n)$ tend vers l'infini.
\end{enumerate}

Nous traitons ces deux cas s\'epar\'ement.

{\bf Cas 1) :}
On applique le th\'eor\`eme de Howe-Moore
\`a la suite $(g^1_n)$ dans le groupe $G_{\mathcal P_1}$ et aux deux
sous-ensembles compacts $\left\{f\right\}$
et $\left\{g^2_n.h \d{ pour } n \in \dr N\right\}$ de $L^2_0(\d{SO}(q,\dr
A)/\d{SO}(q,\dr Q))$. On obtient alors que le coefficient
$\langle f,g_n.h\rangle =\langle f,g^1_n.(g^2_n.h)\rangle $ tend vers $0$ avec
$n$.

{\bf Cas 2) :} On remarque que pour tout nombre premier $p$, l'action de
$\d{SO}(q,\dr
 Z_p)$ commute \`a celle de $\overline {g_{p,n}}$. On en d\'eduit que le
vecteur $\overline {g_{p,n}}.h$ est $\d{SO}(q,\dr Z_p)$-fini et on a :
$$d((\overline {g_{p,n}}.h), p)=d(h,p)\leq d(h)\;.$$ On peut donc appliquer
la troisi\`eme assertion du fait \ref{Oh}. Pour $p$ dans $\mathcal
P_2$, on obtient :
$$\langle f,g_n.h\rangle =\langle f,g_{p,n}.(\overline {g_{p,n}}.h)\rangle \leq
\sqrt{d(f) d(h)}
\xi_p(g_{p,n})\;.$$

On en d\'eduit que le coefficient $\langle f, g_n.h\rangle $ est major\'e par :
$$\sqrt{d(f) d(h)}
\d{inf}_{p \in\mathcal P_2} \xi_p(g_{p,n})\;.$$ Or les deux premi\`eres
assertions
du fait \ref{Oh} impliquent que ce minimum tend vers $0$ quand $g^2_n$
tend vers l'infini.
Ceci termine la preuve de la proposition.
\end{proof}

\subsection{Invariance par un sous-groupe compact ouvert}

Pour appliquer ce th\'eor\`eme, nous devons comprendre comment s'\'ecrit une
fonction dans la d\'ecomposition de $L^2(\GG)$ en $L^2_0(\GG)$ et son
orthogonal. Or les fonctions que nous \'etudierons seront toutes invariantes
par un sous-groupe compact ouvert. Nous fixons donc un tel sous-groupe $U$
compact ouvert.

Notons alors $\La_U$ l'ensemble des caract\`eres unitaires
$U$-invariants de $\La$ et $G_U$ l'intersection de tous les
noyaux des caract\`eres de $\La_U$ :
$$G_U =\cap_{\chi \in \La_U} \d{Ker} \chi \;.$$ On dispose alors du lemme
suivant, en partie tir\'e du lemme 4.1 de \cite{gorodnik-oh-maucourant} :

\begin{lemm}\label{lem:dec}
\begin{enumerate}
        \item L'ensemble $U G^{\infty} \G$ est inclus dans $G_U$.
        \item Le groupe $G_U$ est d'indice fini $N_U$ dans $G$.
        \item Si $f \in L^2(\GG)$ est d\'efinie sur $\pi(G_U)$ et est
$U$-invariante, alors on a :  $$f - \left(\int_{\pi(G_U)} f\d{d}m\right)
{\bf 1}_{\pi(G_U)} \in L^2_0(\GG)$$
        \item Si $f \in L^2(\GG)$ est d\'efinie sur $\pi(G_U)$ et est
$U$-invariante, alors on a la convergence :$$\langle h.f,f\rangle
\xrightarrow[h\in G_U]{h\to \infty} \left(\int_\GG f\d{d}m\right)^2\;.$$
\end{enumerate}
\end{lemm}

\begin{proof}
Le premier point est une cons\'equence de la continuit\'e des
caract\`eres, et du fait que  $G^{\infty}$ est connexe car $\bG$ est connexe et
$\dr R$-anisotrope. On en
d\'eduit le deuxi\`eme car, d'apr\`es la th\'eorie de la r\'eduction (version
ad\'elique) \cite[Th\'eor\`eme 5.1]{platonov-rapinchuk}, l'ensemble $G^{\infty}U
\backslash G / \G$ est fini.

Pour le point 3 : soit $\chi \in \La$. On veut calculer $\langle f,\chi\rangle
$. Dans
un premier temps, si $U$ n'est pas dans le noyau de $\chi$, alors ce produit
scalaire est nul. Ensuite, si $\chi \in \La_U$, alors $\langle f,\chi\rangle =
\int_{\pi(G_U)} f\d{d}m$.

Pour le dernier point : soit $\bar f=f - \left(\int_{\pi(G_U)} f\d{d}m\right)
{\bf 1}_{\pi(G_U)}$. D'apr\`es le point 4, on peut appliquer le th\'eor\`eme
\ref{the:dec} \`a $\bar f$. On a alors : $\langle \bar
f,h.\bar f\rangle \xrightarrow{h \to \infty} 0$. Or on v\'erifie ais\'ement que
quand
$h \in G_U$, on a $\langle \bar f,h.\bar f\rangle =\int_\GG
f(hx)f(x)\d{d}m-\left(\int_\GG
f\d{d}m\right)^2$.
\end{proof}

Nous pouvons maintenant d\'emontrer le th\'eor\`eme \ref{the:appl}.

\section{Dualit\'e}\label{thm}

Nous allons en r\'ealit\'e d\'emontrer un th\'eor\`eme plus pr\'ecis que le
th\'eor\`eme \ref{the:appl}. En effet, dans les hypoth\`eses de ce th\'eor\`eme,
on avait besoin de supposer que $\d{Card}(\G_n)$ tend vers l'infini. Cette
hypoth\`ese est
en pratique difficile \`a v\'erifier. Par exemple dans le cadre unitaire
d\'ecrit dans l'introduction, il faudrait pour appliquer le th\'eor\`eme
\ref{the:appl} conna\^itre a priori un grand nombre de solutions enti\`eres de
l'\'equation $(E_n)$.

\subsection{Le th\'eor\`eme d'\'equidistribution}

Dans le th\'eor\`eme suivant, cette hypoth\`ese est remplac\'ee par
l'hypoth\`ese que les ensembles compacts $G_n\cap G_U$ sont deux \`a deux
distincts. On
remarque que cette hypoth\`ese est a priori plus simple \`a v\'erifier, car
nous n'avons plus besoin de trouver des solutions enti\`eres. Nous
reviendrons l\`a-dessus pour les applications dans les parties \ref{uni} et
\ref{orth}.

Nous utilisons les notations d\'efinies dans la partie \ref{notations} :

\begin{theo}\label{the:appl2}
Soit ${\bf G}$ un $K$-groupe, presque-$K$-simple, connexe tel que le
compl\'et\'e aux places archim\'ediennes $G^{\infty}$ est compact.
Soient $U$ un sous-groupe compact ouvert du groupe des ad\`eles finies $G^f$ et
$(H_n)$ une suite de sous-ensembles compacts bi-$U$-invariants de $G^f$.
Soit $\G$ un sous-groupe arithm\'etique du groupe des ad\`eles $G$ et $\G_n=\G
\cap (G^{\infty}\times H_n)$. On suppose que les ensembles $G_n\cap G_U$ sont
deux \`a deux distincts.

Alors, la projection de $\G_n$ dans $G^\infty$ s'\'equir\'epartit dans
$G^\infty$ :
$$\lim_{n\to \infty} \frac{(\mu \otimes \la) (\GG)}{(\mu \otimes \la)(G_n
\cap G_U)} \sum_{\g \in
\G_n} \delta_{\tau^{\infty}(\g)} = \mu \;;$$

en particulier, $\d{Card}(\G_n) \sim_{n\to \infty} \frac{(\mu \otimes \la)(G_n
\cap G_U)}{(\mu \otimes \la) (\GG)}$
\end{theo}

\begin{rema}
Nous pouvons supposer que le groupe $\bG$ est en r\'ealit\'e
\emph{absolument} presque-simple. En effet, si $\bG$ est
presque-$K$-simple, il existe une extension finie $L$ de $K$ et ${\bf H}$ un
$L$-groupe absolument presque-simple tels que $\bG$ est d\'efini comme la
restriction des scalaires de $L$ \`a $K$ \cite[6.21.ii]{borel-tits}. Mais
alors, par d\'efinition de la restriction des scalaires \cite[Section
2.1.2]{platonov-rapinchuk}, les $K$-points de $\bG$
s' identifient canoniquement aux $L$-points de ${\bf H}$, et le produit des
compl\'et\'es de $\bG$ aux places archim\'ediennes de $K$, c'est-\`a-dire le
groupe $G^\infty$, est isomorphe au produit $H^\infty$ des compl\'et\'es de
${\bf H}$ aux places archim\'ediennes de $L$. Donc pour d\'emontrer le
th\'eor\`eme pour le groupe $\bG$, il suffit de le montrer pour le groupe
absolument presque-simple ${\bf H}$.

Dor\'enavant, nous supposons toujours le groupe $\bG$ absolument
presque-simple et en particulier il v\'erifie les hypoth\`eses du th\'eor\`eme
\ref{the:dec}.
\end{rema}

Avant de prouver ce th\'eor\`eme, nous avons besoin de savoir que, dans les
condition du th\'eor\`eme, le volume des ensembles $G_n\cap G_U$ tend vers
l'infini. Le lemme suivant \'enonce ce r\'esultat bien connu des sp\'ecialistes
:

\begin{lemm}\label{lem:volume}
Soit $C_n$ une suite de parties compactes bi-$U$-invariantes 2 \`a 2 distinctes
de $G^f$.
Alors on a la limite : $\lim_{n \to \infty} \lambda(C_n) =+\infty$
\end{lemm}

Ce lemme est une cons\'equence des
formules de d\'enombrement du volume d'une double-classe pour la d\'ecomposition
de Cartan dans un groupe $p$-adique. On peut trouver de
telles formules dans \cite[paragraphe 1.5]{casselman} ou bien
\cite[7.3]{Gross}. La formule que nous utilisons d\'ecoule directement d'une
formule de d\'enombrement pour le volume des doubles classes modulo un
sous-groupe d'Iwahori dans un groupe $p$-adique \cite[Paragraphe 3.3]{Tits}.

Cependant, comme nous n'avons pas de r\'ef\'erences pr\'ecises et que la
preuve de ce r\'esultat n\'ecessite l'introduction des objets classiques
d'\'etude des groupes semi-simples $p$-adiques, nous renvoyons cette preuve \`a
la partie \ref{partie:volume}.

\subsection{Preuve par dualit\'e}

Nous prouvons maintenant le th\'eor\`eme \ref{the:appl2}.

\begin{proof}[D\'emonstration du th\'eor\`eme \ref{the:appl2}]
Fixons une fonction $\varphi$ continue sur $G^{\infty}$. Nous voulons prouver
la limite suivante :

$$
\lim_{n\to\infty} \frac{(\mu \otimes \la) (\GG)}{(\mu \otimes \la)(G_n\cap G_U)}
\sum_{\g \in
\G_n} \varphi
(\tau^{\infty}(\g)) = \int_{G^{\infty}} \varphi \d{d}\mu \;.
$$

Pour cela, nous d\'efinissons la fonction $f$ sur $G$ en posant
$f(g_{\infty},g_f)=\varphi(g_{\infty})$. Nous posons ensuite
$\displaystyle F_n(g,h)=\sum_{\g \in \G} f(g\g h^{-1}) {\bf 1}_{G_n} (g\g
h^{-1})$.

\begin{rema}
L'introduction de ces fonctions $F_n$ est classique. Elles sont par exemple
utilis\'ees dans \cite[Partie 5]{gorodnik-oh-maucourant}. On peut
remarquer que, au moins dans le cas o\`u $\varphi$ est la fonction constante
\'egale \`a 1, ces fonctions $F_n$ sont exactement celles introduites dans
\cite[Partie 5]{Eskin-MacMullen}. En effet, nous voyons ici le groupe $G$
comme l'espace sym\'etrique $(G\times G)/G$, o\`u $G$ agit sur $G\times G$ par
$g.(g_1,g_2)=(gg_1,g_2g^{-1})$.

Plus g\'en\'eralement, la preuve du th\'eor\`eme 3.1 est un avatar de la
m\'ethode pr\'esent\'ee dans \cite{Eskin-MacMullen}.
\end{rema}

On observe que pour tous $u_1$, $u_2$ dans $U$ et $\g_1$, $\g_2$ dans
$\G$, on a :
\begin{eqnarray}\label{F}
 F_n(u_1 g \g_1, u_2 h \g_2)= F_n(g,h)\;.
\end{eqnarray}
(En effet, on peut
modifier les parties non-archim\'ediennes sans modifier la valeur de $f$). Ainsi
$F_n$ est une fonction continue born\'ee d\'efinie sur $(\GG)^2$, invariante par
l'action \`a gauche de $U\times U$. De plus, on
remarque - en notant $e$ l'\'el\'ement neutre de $G$ - que : $$\displaystyle
F_n(e,e) = \sum_{\g \in \G_n} \varphi (\tau^{\infty}(\g)) \;.$$

La fonction $\varphi$ est continue sur le groupe compact $G^\infty$, elle est
donc
uniform\'ement continue. Ainsi, soient $\varepsilon> 0$ et
$U_{\varepsilon}$ un
voisinage de l'identit\'e dans
$G^{\infty}$ tels que pour tout $u$ et $v \in U_{\varepsilon}$, pour tout $g \in
G^{\infty}$, on a $|\varphi(ugv)-\varphi(g)|\leq \varepsilon$. Notons $\beta$ la
fonction $\frac{1}{\mu(U_{\varepsilon})}
{\bf 1}_{U_{\varepsilon}}$. On pose maintenant : $$\bar \alpha
(g_{\infty},g_f)=\beta(g_{\infty}) {\bf 1}_U(g_f)$$ $$\textrm{et } \displaystyle
\alpha(g) = (\mu\otimes\lambda)(\GG) \sum_{\g \in \G} \bar\alpha(g\g) \;.$$

Par construction, $\bar\alpha$ est une fonction sur $G$, $U$-invariante de
support $U_\varepsilon \times U \subset G_U$ et on a $\int_G \bar \alpha
\d{d}(\mu\otimes\lambda)=1$. On en d\'eduit
que $\alpha$ est une fonction dans $L^2(\GG)$, $U$-invariante, d\'efinie sur
$\pi(G_U)$ et d'int\'egrale par rapport \`a $m$ \'egale \`a 1.

Soit $(x,y)\in
(\GG)^2$, tel que $\alpha(x)\alpha(y)\neq 0$. Alors $x$ et $y$
appartiennent \`a $\pi(U_\varepsilon\times U)$, c'est-\`a-dire qu'ils
s'\'ecrivent
$x=u u_\varepsilon \G$ et $y=v v_\varepsilon \G$ avec $u_\varepsilon$ et
$v_\varepsilon$ dans $U_\varepsilon$, et $u$ et $v$ dans $U$. Ainsi en utilisant
l'\'egalit\'e \ref{F} et le fait qu'appartenir \`a $G_n$ ne d\'epend pas de la
partie archim\'edienne d'un \'el\'ement, on a : $$F_n(x,y) =
F_n(u_\varepsilon,v_\varepsilon)=\sum_{\g\in\G}\varphi(u_\varepsilon\g
v_\varepsilon) {\bf 1}_{G_n}(\g)\;.$$ Par d\'efinition de $U_\varepsilon$, on a
alors : $$|F_n(x,y)-F_n(e,e)|\leq \varepsilon
\d{Card}(\G\cap G_n) \;.$$
Cela implique :
$$|F_n(e,e) - \int_{\GG} \int_{\GG} F_n(x,y) \alpha(x)\alpha(y)
\d{d}m(x)\d{d}m(y)|
\leq \varepsilon \d{Card}(\G\cap G_n) \;.$$

Pour fixer les notations, fixons $X$ un relev\'e de $\GG$ dans $G$. On note
$\tilde \alpha$ le relev\'e de $\alpha$ : $\tilde \alpha = \alpha \circ \pi$, et
on note $\tilde m = \frac{(\mu \otimes
\la)}{(\mu \otimes \la)(X)}$. Enfin notons $I_{n,\varepsilon}$ l'int\'egrale :
$$I_{n,\varepsilon}=\int_{\GG} \int_{\GG} F_n(x,y) \alpha(x)\alpha(y)
\d{d}m(x)\d{d}m(y)$$ On fait alors
le calcul suivant :

\begin{eqnarray*}
I_{n,\varepsilon} & = & \int_{X}
\int_{X} F_n(x,y) \tilde\alpha(x)\tilde\alpha(y) \d{d}\tilde m(x)\d{d}\tilde
m(y) \\
 & = & \int_{X} \int_{X} \sum_{\g \in \G} f(x\g y^{-1}) {\bf 1}_{G_n} (x\g
y^{-1})
\tilde\alpha(x)\tilde\alpha(y) \d{d}\tilde m(x)\d{d}\tilde m(y)
\end{eqnarray*}
On fait pour tout $\g \in \G$ le changement de variable $x\g=g$, ce qui permet
d'obtenir :
\begin{eqnarray*}
I_{n,\varepsilon} & = & \int_G \int_X f(gy^{-1}) {\bf 1}_{G_n}(g y^{-1})
\tilde\alpha(g)\tilde\alpha(y)\d{d}\tilde m(g)\d{d}\tilde m(y)
\end{eqnarray*}
On fait maintenant le changement de variables $h=gy^{-1}$ pour tout $g\in G$.
On obtient finalement :
\begin{eqnarray*}
I_{n,\varepsilon}
 & = & \int_G f(h){\bf 1}_{G_n}(h) \int_X
\tilde\alpha(hy)\tilde\alpha(y)\d{d}\tilde m(y)\d{d}\tilde m(h)
\end{eqnarray*}

Consid\'erons le coefficient matriciel $\int_X
\tilde\alpha(hy)\tilde\alpha(y)\d{d}\tilde m(y)= \langle h.\alpha,\alpha\rangle
$ de la
repr\'esentation de $G$ dans $L^2(\GG)$.

On remarque tout d'abord que si $h$ n'appartient pas \`a $G_U$, et
$x \in \GG$ appartient \`a $\pi(G_U)$, l'\'el\'ement $hx$ de $\GG$ n'appartient
pas \`a $\pi(G_U)$ ; comme le support de $\alpha$ est inclus dans $\pi(G_U)$, on
a alors $(h.\alpha)(x)=\alpha(hx)=0$. Cela signifie que le coefficient
$\langle h.\alpha,\alpha\rangle $ est nul d\`es que $h\not\in G_U$. C'est ainsi
qu'appara\^issent les ensembles $G_n\cap G_U$ :
\begin{eqnarray*}
I_{n,\varepsilon}
 & = & \int_G f(h){\bf 1}_{G_n\cap G_U}(h) \langle h.\alpha,\alpha\rangle
\d{d}\tilde m(h)
\end{eqnarray*}

Ensuite, nous appliquons le th\'eor\`eme de d\'ecroissance \`a ce
coefficient matriciel. On utilise \`a nouveau le fait que $\alpha$ est d\'efinie
sur $\pi(G_U)$ et aussi qu'elle est $U$-invariante et d'int\'egrale $1$. En
effet, cela permet d'appliquer le point 4 du lemme \ref{lem:dec} et on obtient :
$$\langle h.\alpha,\alpha\rangle  \xrightarrow[h \in G_U]{h \to \infty}
1\;.$$
Ainsi il existe une partie compacte $C$ de $G_U$, tel que
pour tout $h \in G_U - C$, on a la majoration $|\langle h.\alpha,\alpha\rangle
-1 | \leq
\varepsilon$. De plus, par d\'efinition de
$f$, on a :
$$\int_G f(h){\bf 1}_{G_n\cap G_U}(h)\d{d}\tilde m(h)=\frac{(\mu
\otimes \la)(G_n\cap G_U)}{(\mu\otimes\la)(\GG)}
\int_{G^{\infty}}\varphi \d{d}\mu$$
Et alors, pour $n$ suffisamment
grand, il existe une constante $A$ telle que :
\begin{eqnarray*}
|I_{n,\varepsilon}- \frac{(\mu
\otimes \la)(G_n\cap G_U)}{(\mu\otimes\la)(\GG)}
\int_{G^{\infty}}\varphi \d{d}\mu| & \leq  \varepsilon\frac{(\mu
\otimes \la)(G_n\cap
G_U)}{(\mu\otimes\la)(\GG)}\int_{G^{\infty}}\varphi \d{d}\mu +A \;.\\
\end{eqnarray*}
On utilise maintenant le fait que le volume $(\mu\otimes\lambda)(G_n\cap G_U)$
tend vers l'infini (voir le
lemme \ref{lem:volume}) pour obtenir :
\begin{eqnarray*}
|I_{n,\varepsilon}- \frac{(\mu
\otimes \la)(G_n\cap G_U)}{(\mu\otimes\la)(\GG)}
\int_{G^{\infty}}\varphi \d{d}\mu|
& \leq & \varepsilon\frac{2(\mu
\otimes \la)(G_n\cap
G_U)}{(\mu\otimes\la)(\GG)} \int_{G^\infty} \varphi \d{d}\mu \; .\\
\end{eqnarray*}
On en d\'eduit que pour $n$ grand, on a l'in\'egalit\'e :
\begin{flushleft}
$\displaystyle |F_n(e,e)-\frac{(\mu
\otimes \la)(G_n\cap G_U)}{(\mu\otimes\la)(\GG)} \int_{G^{\infty}}\varphi
\d{d}\mu| \leq$ \end{flushleft}
\begin{flushright} $\displaystyle \left(\frac{2(\mu
\otimes \la)(G_n\cap
G_U)}{(\mu\otimes\la)(\GG)}\int_{G^{\infty}}\varphi
\d{d}\mu+\d{Card}(\G\cap G_n)\right)\varepsilon$ .\end{flushright}
C'est \`a dire qu'on a pour tout $\varepsilon> 0$ :
\begin{eqnarray*}
\limsup_{n\to \infty} \frac{(\mu \otimes \la) (\GG)}{(\mu
\otimes \la)(G_n \cap
G_U)}(F_n(e,e)-\varepsilon \d{Card}(\G\cap G_n)) & \leq &
(1 + 2\varepsilon)\int_{G^{\infty}}\varphi \d{d}\mu\\
\liminf_{n\to \infty} \frac{(\mu \otimes \la) (\GG)}{(\mu
\otimes \la)(G_n \cap
G_U)}(F_n(e,e)+\varepsilon \d{Card}(\G\cap G_n)) &
\geq & (1 - 2\varepsilon)\int_{G^{\infty}}\varphi \d{d}\mu\\
\end{eqnarray*}

On conclut en deux \'etapes : tout d'abord, on applique les deux
in\'egalit\'es pr\'ec\'edentes \`a $\varphi=1$, auquel cas
$F_n(e,e)=\d{Card}(\G\cap G_n)$. On en d\'eduit (en faisant tendre $\varepsilon$
vers 0) que :
$$\frac{(\mu \otimes \la) (\GG)}{(\mu
\otimes \la)(G_n \cap G_U)}\d{Card}(\G\cap G_n)
\xrightarrow{n\to \infty} 1 \;.$$ Enfin, on utilise ce r\'esultat pour traiter
le
cas g\'en\'eral. Pour tout $\varepsilon > 0$, on a :

\begin{eqnarray*}
\limsup_{n\to \infty} \frac{(\mu \otimes \la) (\GG)}{(\mu
\otimes \la)(G_n \cap
G_U)}F_n(e,e) & \leq & (1+2\varepsilon)\int_{G^{\infty}}\varphi
\d{d}\mu + \varepsilon\\
\liminf_{n\to \infty} \frac{(\mu \otimes \la) (\GG)}{(\mu\otimes\la)(G_n \cap
G_U)}F_n(e,e) & \geq & (1-2\varepsilon)\int_{G^{\infty}}\varphi \d{d}\mu -
\varepsilon\\
\end{eqnarray*}

On peut maintenant faire tendre $\varepsilon$ vers 0 pour obtenir la
limite voulue. Ainsi, le th\'eor\`eme \ref{the:appl2} est prouv\'e, et donc
aussi le th\'eor\`eme \ref{the:appl}.
\end{proof}

\section{Cas des groupes unitaires}\label{uni}

Nous prouvons dans cette partie le th\'eor\`eme \ref{the:equiuni}, et donc le
th\'eor\`eme \ref{coro:existenceuni}. Nous reprenons
les notations donn\'ee dans l'introduction.

\begin{proof}[D\'emonstration du th\'eor\`eme \ref{the:equiuni}]
On applique le th\'eor\`eme \ref{the:appl2} dans le cas suivant : le groupe
$\bG$ est le $\dr Q$-groupe $\d{SU}(h)$. On note qu'il v\'erifie bien les
hypoth\`eses du th\'eor\`eme et de plus qu'il est simplement connexe
\cite[Paragraphe 2.3.3]{platonov-rapinchuk}. On choisit
$\G=\d{SU}(h,\dr Q)$, et
$U$ le produit pour $p$ premier des sous-groupes compacts ouverts $\d{SU}(h,\dr
Z_p)$.

Pour tout $p$ premier et $r$ entier, on note $L_{p^r}$ le sous-ensemble
de $\bG(\dr Q_p)$ compos\'e des matrices telles que le supremum de la valeur
absolue des coefficients est $p^r$. Enfin pour tout $n$ entier avec
$\displaystyle n=\prod_{p \; \d{premier}} p^{\nu_p(n)}$, on note $\displaystyle
H_n=\prod_{p \; \d{premier}} L_{p^{\nu_p(n)}}$.

On remarque alors que un entier $n$ appartient \`a $\mathcal U_l(H)$ si et
seulement si $H_n$ est non vide. De plus, pour toute matrice $M$ de $\d{M}(h,\dr
Z[i])$, $M$ est dans $\mathcal T(n,H,\dr Z)$ si et seulement si
$\frac{1}{n}M$ appartient \`a $\bG(\dr R)\times H_n$. C'est \`a dire que les
ensembles $\G_n$ d\'efinis dans l'\'enonc\'e du th\'eor\`eme sont bien \'egaux
\`a $\G\cap G_n$.

Pour pouvoir appliquer le th\'eor\`eme \ref{the:appl2}, il ne reste plus qu'\`a
prouver que les ensembles $G_n\cap G_U$ sont distincts. Or on remarque que,
pour $n$ dans $\mathcal U_l(H)$, les $H_n$ sont disjoints donc distincts. Et il
en est de m\^eme des $G_n$. Il suffit alors de v\'erifier que $G_U=G$.
Le lemme suivant appliqu\'e \`a $L=G_U$ permet de conclure :

\begin{lemm}\label{lem:simpconn}
Soit $\bG$ un $\dr Q$-groupe, presque-$\dr Q$-simple et
simplement connexe. Alors tout sous-groupe $L$ ferm\'e normal, contenant
$\bG(\dr Q)$ et
d'indice fini dans $\bG(\dr A)$ est \'egal \`a $\bG(\dr A)$.
\end{lemm}

\begin{proof}
En effet, soit $p$ un nombre premier tel que $\bG(\dr Q_p)$
est isotrope (un tel $p$ existe, c'est m\^eme le cas pour presque tous les $p$
\cite[Th\'eor\`eme 6.7]{platonov-rapinchuk}). Alors, par hypoth\`ese de simple
connexit\'e \cite[Paragraphe 7.2]{platonov-rapinchuk}, le groupe $\bG(\dr Q_p)$
est engendr\'e par ses
\'el\'ements unipotents et notamment ne contient pas de
sous-groupe d'indice fini diff\'erent de lui-m\^eme. Donc le groupe
$L \cap \bG(\dr Q_p)$ (ici, on a plong\'e de fa\c con naturelle $\bG(\dr Q_p)$
dans $\bG(\dr A)$) est \'egal \`a $\bG(\dr Q_p)$.

Ensuite, par la propri\'et\'e d'approximation forte \cite[Th\'eor\`eme
7.12]{platonov-rapinchuk}, le produit $\bG(\dr Q_p)\bG(\dr Q)$ est dense dans
$\bG(\dr A)$. Donc $L=\bG(\dr A)$.
\end{proof}

Cela finit la preuve des
th\'eor\`emes \ref{the:equiuni} et
\ref{coro:existenceuni}.
\end{proof}

\section{Cas des groupes orthogonaux}\label{orth}

Nous nous int\'eressons dans cette partie au cas des groupes orthogonaux.
Rappelons tout d'abord que les groupes orthogonaux ne sont pas simplement
connexes \cite[Paragraphe 2.3.2, Proposition 2.14]{platonov-rapinchuk},
donc le lemme \ref{lem:simpconn} ne s'applique pas.

De fait, la question du passage du local au global pour les formes quadratiques
\`a coefficients entiers a \'et\'e beaucoup \'etudi\'e et le th\'eor\`eme
cit\'e au d\'ebut de ce texte donne une r\'eponse satisfaisante dans les cas
o\`u le rang vaut au moins $5$. Esquissons, dans les grandes lignes, la
strat\'egie pour prouver ce th\'eor\`eme :

On dit qu'une forme quadratique $q$ d\'efinie positive
repr\'esente (resp. repr\'esente localement) un entier naturel $n$ si $n$
appartient \`a $q(\dr Z)$ (resp. \`a $q(\dr Z_p)$ pour tout $p$ premier).
Nous rappelons aussi la d\'efinition du genre d'une forme quadratique $q$.
C'est l'ensemble des formes quadratiques \'equivalentes \`a $q$ \`a la fois sur
$\dr Q$ et sur tous les $\dr Z_p$ :

\begin{defi}\label{genre}
Soient $q$ et $q'$ des formes quadratiques \`a coefficients entiers
de rang $k$, associ\'ees respectivement aux matrices $Q$ et $Q'$. On dit
qu'elles sont dans le m\^eme genre si elles v\'erifient les propri\'et\'es
suivantes :
\begin{itemize}
 \item il existe $g \in \d{GL}(k,\dr Q)$ tel que $Q'={}^tgQg$.
 \item pour tout nombre premier $p$, il existe un \'el\'ement $g_p \in
\d{GL}(k,\dr
Z_p)$
tel que $Q'={}^tg_pQg_p$
\end{itemize}
\end{defi}

On prouve alors que si un entier est repr\'esent\'e localement par une forme
quadratique $q$, il existe une forme dans le genre de $q$ qui le repr\'esente
\cite[Chap. 9, Th\'eor\`eme 1.3]{cassels}. Ensuite,
on d\'emontre, du moins
quand le rang est au moins $5$, que toutes les formes d'un m\^eme genre
repr\'esentent les m\^emes entiers suffisamment grands. Cette derni\`ere \'etape
ne fonctionne pas en toute g\'en\'eralit\'e en rang $4$, et pas du tout en rang
$3$, o\`u il faut introduire le concept de
genre-spin. Nous ne d\'ecrirons pas davantage ces th\'eories et renvoyons \`a
l'article de W. Duke \cite{duke-survey} pour une pr\'esentation historique de
ce probl\`eme, ainsi qu'\`a l'article de W. Duke et R. Schulze-Pillot
\cite{duke-schulze} pour l'analyse du cas de rang 3.

\bigskip

Pr\'esentons maintenant les r\'esultats que nous obtenons : fixons une
forme quadratique $q$ rationnelle d\'efinie positive de rang $k\geq 3$, de
matrice associ\'ee $Q$.
Pour tout entier $n$ et pour $A=\dr Z$ ou $\dr Z_p$, notons $\mathcal
S(n,q,A)$ l'ensemble des matrices $M\in \d{M}(k,A)$ de d\'eterminant $n^k$
telles que $\frac{1}{n}M$ est de d\'enominateur $n$, c'est-\`a-dire :
\begin{itemize}
\item les coefficients de $M$ sont premiers entre eux,
\item la matrice $M$ est solution de l'\'equation $(F_n)$ : ${}^t MQM= n^2 Q$.
\end{itemize}

Notons $\mathcal R_l(q)$ l'ensemble des entiers $n$ tels que pour tout nombre
premier $p$, $\mathcal S(n,q,\dr Z_p)$ est non vide. On notera de plus
$\mathcal
R_{\d{genre}}(q)$ l'ensemble des entiers $n$ tels qu'il existe une forme $q'$
dans
le genre de $q$ avec $\mathcal S(n,q', \dr Z)$ non vide.

De la m\^eme mani\`ere que dans le cas unitaire, nous cherchons des matrices
dans $\mathcal S(n,q,\dr Z)$, et \`a comprendre l'image de cet ensemble dans
$\d{SO}(q,\dr R)$.

Dans le cas des formes de rang au moins $5$, nous prouvons avec le
th\'eor\`eme \ref{the:equiorth} que si $n$ est dans $\mathcal R_l(q)$ avec en
plus la
condition que $n$ est premier \`a un certain entier fix\'e, et que $n$ est
suffisamment grand, alors $\mathcal S(n,q,\dr Z)$ est non vide, et son image
dans $\d{SO}(q,\dr R)$ s'\'equir\'epartit vers la mesure de Haar. Le concept de
genre n'intervient pas dans cette partie.

Ensuite, nous essayons de suivre une strat\'egie parall\`ele \`a celle
\'evoqu\'ee plus haut. Nous prouvons, cette fois sans restriction sur le rang,
que, si $n$ est dans $\mathcal R_{\d{genre}}(q)$ et suffisamment grand, alors
$\mathcal S(n,q,\dr Z)$ est non vide, et son image s'\'equir\'epartit dans
$\d{SO}(q,\dr R)$. C'est l'objet du th\'eor\`eme \ref{the:genre}.

\subsection{Formes de rang sup\'erieur \`a 5}

Commen\c cons par le cas du rang sup\'erieur \`a 5 :

\begin{theo}\label{the:equiorth}
Soient $k \geq 5$, $q$ une $\dr Q$-forme quadratique d\'efinie positive
sur $\dr
Q^k$ et $\mu$ la probabilit\'e de Haar sur $\d{SO}(q,\dr R)$. Pour tout entier
$n$, notons $\G_n$ l'ensemble des matrices de $\d{SO}(q,\dr Q)$ de
d\'enominateur
$n$.

Alors il existe un entier $N$ tel que quand $n$ tend vers l'infini et que $n$
est premier \`a $N$, les $\G_n$
s'\'equir\'epartissent dans $\d{SO}(q,\dr R)$, c'est-\`a-dire :

$$\frac{1}{\d{Card} (\G_n)} \sum_{\g \in \G_n} \delta_{\g} \xrightarrow[n
\textrm{ premier \`a } N]{n \to \infty} \mu\; .$$
\end{theo}

\begin{proof}

Nous allons \`a nouveau appliquer le th\'eor\`eme \ref{the:appl2} en
consid\'erant le groupe $\bG=\d{SO}(q)$, qui en v\'erifie bien les hypoth\`eses.
On pose pour tout nombre premier $p$, $U_p=\d{SO}(q,\dr Z_p)$, et on
note $U$ le produit sur $p$ des $U_p$.

Soit, pour $p$ premier et $m$ entier, $\tilde H_{p^m}$ l'ensemble des
matrices de $\d{SO}(q,\dr Q_p)$ telles que le maximum de la norme $p$-adique des
coefficients est $p^m$. Maintenant, pour un entier $n$, pour tout nombre
premier $p$,
on note $\nu_p(n)$ la valuation $p$-adique de $n$. On pose alors $H_n$ le
produit sur les $p$ premiers des $\tilde H_{p^{\nu_p(n)}}$. Ces ensembles $H_n$
sont bi-$U$-invariants et deux \`a deux disjoints.

On v\'erifie alors que, pour tout $n$, l'ensemble $\G_n$ est exactement
$\d{SO}(q,\dr Q) \cap G_n$.

\bigskip

Il nous faut maintenant d\'eterminer un ensemble fini $F$ de nombres premiers
tel que si $n$ est premier aux \'el\'ements de $F$, alors $G_n \cap G_U$ est
non vide. Il suffira alors de choisir pour $N$ le produit des nombres premiers
dans $F$. Pour cela, on commence par un lemme de r\'eduction des formes
quadratiques :

\begin{lemm}\label{rang2}
Soient $k \geq 5$, $q$ une $\dr Q$-forme quadratique d\'efinie positive sur $\dr
Q^k$. Alors il existe un ensemble fini $F$ de nombres
premiers tels que pour tout nombre premier $p$ en dehors de $F$, on a :

la forme $q$ est conjugu\'ee par une matrice de $\d{\d{GL}}(k,\dr
Z_p)$ \`a une forme quadratique de la forme $q'(x_1,\ldots ,x_k)=x_1 x_2 + x_3
x_4 + q''(x_5,\ldots, x_k)$.
\end{lemm}

\begin{proof}
La forme $q$ est conjugu\'ee par $A \in \d{\d{GL}}(n,\dr Q)$ \`a une forme
quadratique diagonale $\bar q$. Soit $F$ l'ensemble des nombres premiers $p$
tels que ou bien $p=2$ ou bien $A$ n'appartient pas \`a $\d{GL}(k,\dr Z_p)$ ou
bien $\bar
q$ n'est pas \`a coefficients dans $\dr Z_p^*$ pour la base canonique. Alors,
pour tout $p \not\in F$, $q$ est conjugu\'e sur $\d{GL}(k,\dr Z_p)$ \`a une
forme quadratique diagonale \`a coefficients dans $\dr Z_p^*$.

Soit $r$ une forme quadratique diagonale en trois variables \`a coefficients
dans $\dr Z_p^*$. L'ensemble $\dr Z_p^* /(\dr Z_p^*)^2$ est compos\'e de deux
\'el\'ements. Donc, $r$ est \'equivalente sur $\dr Z_p^*$ \`a
$\alpha(x^2+y^2)+\beta z^2$, pour un
$\alpha \in \dr Z_p^*$ et un $\beta \in \dr Z_p^*$. De plus tout \'el\'ement de
$\dr Z_p^*$ s'\'ecrit comme somme de deux carr\'es par le lemme de Hensel
\cite[Chap II, section 2.2]{serre} (on utilise ici $p\neq 2$), donc il existe
$a$ et $b$ dans $\dr Z_p^*$ tels que $\alpha (a^2 + b^2)=-\beta$. Alors, dans
la base $\left((a,b,1),\beta^{-1}(-b,a,1),(a-b,a+b,1)\right)$, la forme $r$
s'\'ecrit $xy+\beta z^2$. Par construction la matrice qui conjugue $r$ \`a cette
derni\`ere forme est bien dans $\d{GL}(k,\dr Z_p)$.

Pour la forme $q$ diagonale en au moins cinq variables, on peut appliquer le
proc\'ed\'e ci-dessus deux fois pour obtenir le r\'esultat voulu : pour tout
nombre premier $p$ impair, toute forme quadratique $r$ diagonale
sur $\dr Q_p^5$ \`a coefficients dans $\dr Z_p^*$  est \'equivalente sur $\dr
Z_p$ \`a une forme quadratique $y_1 y_2 + +y_3y_4 + \alpha y_5^2$ pour un
$\alpha \in \dr Z_p^*$.
\end{proof}

Pour tout $p \not\in F$, on note  $\varphi_p$ l'isomorphisme entre $\d{SO}(q,\dr
Q_p)$ et $\d{SO}(q',\dr Q_p)$ donn\'e par le lemme pr\'ec\'edent et pour
tout $m$
entier, $J_{p^m}$ l'ensemble des matrices de
$\d{SO}(q',\dr Q_p)$ telles que le maximum de la norme $p$-adique des
coefficients est $p^m$. Comme le changement de base est \`a coefficients dans
$\dr Z_p$, on obtient imm\'ediatement le lemme  :

\begin{lemm}
Pour tout $p\not\in F$, pour tout $m\in \dr N$, on a $\varphi_p(\tilde
H_{p^m})=J_{p^m}$.
\end{lemm}

Rappelons que $G_U$ est d\'efini comme l'intersection des noyaux de l'ensemble
$\Lambda_U$ des caract\`eres $U$ et $\G$-invariants de $G$.
Soit $n$ un entier. On veut d\'emontrer que $G_n\cap G_U$ est non vide.
Supposons
qu'on dispose de $g \in G$ et $u \in U$ tel que $g u g^{-1}$ soit un
\'el\'ement de $G_n$. Alors, pour tout $\lambda \in \Lambda^U$, $\lambda
(gug^{-1})=1$, et donc $gug^{-1} \in G_n \cap G_U$.

On note $\displaystyle N=\prod_{p\in F}p$. On va prouver que l'entier $N$
 convient pour le th\'eor\`eme \ref{the:equiorth} : il
suffit de d\'emontrer que pour tout entier $n$ premier \`a $N$, on peut
trouver une paire $(g,u) \in G\times U$ tel que $gug^{-1}$ est dans $G_n$. Soit
donc $n$ un entier premier \`a $N$.

D'apr\`es la d\'efinition de $H_n$, il
suffit de trouver, pour tout $p$ premier et $m=\nu_p(n)$ entier, une paire
$(g,u) \in \d{SO}(q,\dr Q_p)\times U_p$ telle que $gug^{-1}$ appartient
 \`a $\tilde H_{p^m}$. Si $m=0$, ce qui est le cas notamment si $p \in
F$, il suffit de trouver une matrice dans $\d{SO}(q,\dr Z_p)$ : la matrice
identit\'e convient. Il reste \`a traiter le cas $\nu_p(n)\neq 0$, pour
lequel on sait que $p \not \in F$. Donc on peut appliquer les deux lemmes
pr\'ec\'edents pour l'isomorphisme $\varphi_p$. Le th\'eor\`eme
est alors une cons\'equence du lemme suivant :

\begin{lemm}
Soient $p$ un nombre premier impair, $m \in \dr N$, $k \geq 5$ et $q'$ une forme
quadratique sur $\dr Q_p^k$ de la forme $q'(x_1,\ldots,x_k)=x_1 x_2 +
x_3x_4+q''(x_5,\ldots x_k)$.

Alors il existe $g \in \d{SO}(q',\dr Q_p)$, et $u \in \d{SO}(q',\dr Z_p)$ tel
que
$gug^{-1}$ appartient \`a $J_{p^m}$.
\end{lemm}

\begin{proof}
Il suffit de prendre les matrices :
$$g = \begin{pmatrix} p^{m} &0&0&0&0\\
0&p^{-m}&0&0&0\\0&0&1&0&0\\ 0&0&0&1&0 \\ 0&0&0&0& I_{k-4}\end{pmatrix} \textrm{
et } u = \begin{pmatrix}
1&0&0&0&0\\
0&1&1&0&0\\0&0&1&0&0\\ -1&0&0&1&0\\ 0&0&0&0& I_{k-4}\end{pmatrix}$$
\end{proof}

Le th\'eor\`eme \ref{the:equiorth} est donc bien prouv\'e.
\end{proof}

De m\^eme que dans le cas unitaire, on obtient comme
corollaire un r\'esultat d'existence :

\begin{coro}\label{coro:existenceorth}
Soient $k \geq 5$, $Q \in \d{M}(k,\dr Q)$ une matrice sym\'etrique,
d\'efinie positive.
Alors il existe deux entiers $N$ et $n_0$ tels
qu'on a pour tout entier $n\geq n_0$ :

L'entier $n$ est premier \`a $N$ implique que $\mathcal S(n,Q,\dr Z)$ est non
vide.
\end{coro}

\subsection{Lien avec le genre}

Voil\`a l'\'enonc\'e qui exprime que pour des formes dans le m\^eme genre,
l'ensemble des d\'enominateurs de matrices rationnelles dans leur groupe
orthogonal sont les m\^emes, du moins pour des entiers suffisamment grands :

\begin{theo}\label{the:genre}
Soient $k \geq 3$, $q$ et $q'$ deux formes quadratiques d\'efinies positives
du m\^eme genre.

Alors, pour $n$ suffisamment grand, s'il existe une matrice rationnelle de
d\'enominateur $n$ dans $\d{SO}(q',\dr Q)$, il en existe une dans $\d{SO}(q,\dr
Q)$.
\end{theo}

\begin{proof}
On se place exactement dans le cadre de la preuve du th\'eor\`eme
\ref{the:equiorth} : on consid\`ere \`a nouveau le
groupe $\bG=\d{SO}(q)$. On pose
pour tout nombre premier $p$, $U_p=\d{SO}(q,\dr Z_p)$, et $U$ le produit sur $p$
des
$U_p$.

On d\'efinit encore, pour $p$ premier et $m$ entier, $\tilde H_{p^m}$ l'ensemble
des
matrices de $\d{SO}(q,\dr Q_p)$ telles que le maximum de la norme $p$-adique des
coefficients est $p^m$. Maintenant, pour un entier $n$, pour tout nombre
premier $p$,
on note $\nu_p(n)$ la valuation $p$-adique de $n$. On pose alors $H_n$ le
produit sur les $p$ premiers des $\tilde H_{p^{\nu_p(n)}}$. Ces ensembles $H_n$
sont bi-$U$-invariants et deux \`a deux disjoints.

\bigskip

Soit maintenant $n$ un entier tel qu'il existe une matrice $\gamma$ de
d\'enominateur $n$ dans $\d{SO}(q',\dr Q)$. Pour prouver le th\'eor\`eme,
selon la m\'ethode d\'ej\`a vue, il nous suffit de prouver qu'alors $G_n\cap
G_U$ est non vide. On note $Q$ et $Q'$ les matrices associ\'ees \`a $q$ et $q'$.
Soient $g_{\dr Q}$ la matrice rationnelle conjuguant $Q$ et $Q'$ et, pour tout
nombre premier $p$, $g_p$ la
matrice de $\d{GL}(k,\dr Z_p)$  conjuguant $Q$ et $Q'$. On note
$g$ l'\'el\'ement $(g_{\dr Q},(g_p)_{p \; \d{premier}})$ de $\d{GL}(k,\dr A)$
(ici,
$g_{\dr Q}$ est
vu comme un \'el\'ement de $\d{GL}(k,\dr Q)\subset \d{GL}(k,\dr R)$).
Consid\'erons
l'\'el\'ement $g \gamma g^{-1}$ de $\d{GL}(k,\dr A)$. Alors par
d\'efinition de $g$, c'est un \'el\'ement de $\d{SO}(q,\dr A)$. On veut
d\'emontrer qu'il est dans $G_n\cap G_U$.

Pour cela, commen\c cons par remarquer que pour tout nombre premier $p$, $g_p$
est dans
$\d{GL}(k,\dr Z_p)$. Ainsi on ne change pas la norme $p$-adique de $\gamma$
en le conjuguant par $g_p$. Donc l'\'el\'ement $g \gamma g^{-1}$ est
bien dans $G_n$.

Soit ensuite $\lambda$ un caract\`ere de $\Lambda_U$,
c'est-\`a-dire $U$ et
$\G$-invariant. On veut d\'emontrer que $\lambda(g \gamma g^{-1})$ vaut 1.
D\'efinissons sur $\d{SO}(q',\dr A)$ le caract\`ere $\lambda '$ par : pour tout
$h
\in \d{SO}(q',\dr A)$, $\lambda ' (h) = \lambda (g_{\dr Q} h g_{\dr Q}^{-1})$
(ici
$g_{\dr Q}$ est vu comme
l'\'el\'ement rationel de $\d{GL}(k,\dr A)$ dont chaque composant dans
$\d{GL}(k,\dr R)$ et les $\d{GL}(k,\dr Q_p)$ est la matrice $g_{\dr Q}$). Comme
$g_{\dr
Q}$ est une
matrice rationnelle, $\lambda'$ est $\d{SO}(q',\dr Q)$-invariant.
Or on a $\lambda(g \gamma g^{-1})= \lambda' (g_{\dr Q}^{-1}g
\gamma g^{-1}g_{\dr Q})$. De plus, par construction, $g^{-1}g_{\dr Q}$ est un
\'el\'ement
de $\d{SO}(q',\dr A)$ et $\gamma$ est un \'el\'ement de $\d{SO}(q',\dr Q)$.
Donc, on a d\'emontr\'e le r\'esultat voulu :
$$\lambda(g \gamma g^{-1})=
\lambda' (g_{\dr Q}^{-1}g\gamma g^{-1}g_{\dr Q})=\lambda'(\gamma)=1$$

Cela termine la preuve : il suffit d'appliquer le th\'eor\`eme \ref{the:appl2}.
\end{proof}

\subsection{Application \`a la forme canonique}

Dans cette section, on applique nos r\'esultats au cas de la forme quadratique
canonique. On note $q_k$ la forme quadratique canonique $x_1^2 + \ldots
+x_k^2$ sur $\dr Z^k$. La strat\'egie est la m\^eme, mais la diff\'erence
notable est qu'on sait construire en rang $3$ des matrices rationnelles de tout
d\'enominateur impair dans $\d{SO}(q_3,\dr Q)$, donc nous n'avons plus de
probl\`eme avec le groupe $G_U$.

\begin{coro}\label{coro:appl}
Soient $k \geq 3$ et $\mu$
la probabilit\'e de Haar sur $\d{SO}(q_k,\dr R)$. Pour tout entier $n$,
on note $\G_n$ l'ensemble des
matrices de $\d{SO}(q_k,\dr Q)$ de d\'enominateur $n$.

Alors quand $n$ tend vers l'infini et est impair, les $\G_n$
s'\'equir\'epartissent dans $\d{SO}(q_k,\dr R)$, c'est-\`a-dire :

$$\frac{1}{\d{Card} (\G_n)} \sum_{\g \in
\G_n} \delta_{\g} \xrightarrow[n \textrm{ impair}]{n \to \infty} \mu \; .$$
\end{coro}

\begin{proof}
Commen\c cons par le lemme suivant :

\begin{lemm}\label{presence}
Pour tout entier impair $n$, il existe une matrice de d\'enominateur $n$ dans
$\d{SO}(q_k,\dr Q)$.
\end{lemm}

\begin{proof}
Il suffit de le faire pour $\d{SO}(q_3,\dr Q)$. Soit alors $n$ un entier impair,
et
$x=a+ib+jc+kd$ un quaternion \`a coefficients entiers premiers entre eux de
norme $n$.

Consid\'erons la matrice de l'action de $x$ par conjugaison sur les quaternions
purs. Elle s'\'ecrit :
$$\frac{1}{n}\begin{pmatrix}
a^2+b^2-c^2-d^2 & 2(bc-ad) & 2(ac+bd)\\
2(ad+bc) & a^2-b^2+c^2-d^2 & 2(cd-ab)\\
2(bd-ac) & 2(ab+cd) & a^2-b^2-c^2+d^2\\
\end{pmatrix}
$$

On v\'erifie alors qu'elle est bien de d\'enominateur $n$ (il n'y a pas de
simplification possible). De plus, par construction c'est une matrice de
$\d{SO}(q_3,\dr Q)$.
\end{proof}

On en d\'eduit (avec les notations des preuves pr\'ec\'edentes) que pour tout
$n$ impair, $G_n\cap G_U$ est non vide, car il contient une matrice rationnelle.
On en d\'eduit imm\'ediatement le
r\'esultat du corollaire, comme application du th\'eor\`eme \ref{the:appl2}.
\end{proof}

\section{Volume des doubles classes dans la d\'ecomposition de
Cartan}\label{partie:volume}

Nous donnons ici une preuve du lemme \ref{lem:volume}. Soient donc $\bG$ un
groupe alg\'ebrique d\'efini sur $K$, $U$ un sous-groupe compact ouvert
de $G^f$ et $\la$ la mesure de Haar sur $G^f$ telle que $\la(U)=1$. On veut
prouver que si $g$ tend vers l'infini dans $G^f$, $\la(UgU)$ tend aussi vers
l'infini. Remarquons que localement ce r\'esultat est bien connu et on sait
exprimer le volume des doubles classes modulo un sous-groupe d'Iwahori tr\`es
simplement en
fonction de la longueur pond\'er\'ee sur le groupe de Weyl affine. Cependant
ici $U$ est un sous-groupe compact du groupe des ad\`eles donc la projection de
$U$ sur
presque toutes les places est un sous-groupe compact maximal dans lequel
un sous-groupe d'Iwahori est d'indice de plus en plus grand. Il nous faut donc
contr\^oler ce
ph\'enom\`ene. C'est l'objet du lemme \ref{vol} qui ne pr\'esente pas de
difficult\'es mais n\'ecessite l'introduction du vocabulaire d'\'etude
des groupes r\'eductifs sur
les corps locaux.

\bigskip

On remarque tout d'abord l'\'egalit\'e :
$$\la(UgU)=\d{Card}\left(UgU/U\right)=\d{Card}\left(U/(U\cap
gUg^{-1})\right)\,.$$
De
plus, on peut changer de sous-groupe compact ouvert en vertu du lemme suivant :

\begin{lemm}\label{lem:changer}
Si $V \subset U$ est un autre sous-groupe compact ouvert de $G^f$, il existe
une constante $c> 1$ telle que pour tout $g$ dans $G$, on a :
$$\frac{1}{c}\,\d{Card}(VgV) \leq \d{Card}(UgU) \leq {c}\,\d{Card}(VgV)\;.$$
\end{lemm}

\begin{proof}
Soit $c$ l'indice $[U:V]$. D'une part, on a $$\d{Card } UgU/U \geq \d{Card }
VgU/U=\d{Card } V/(V\cap gUg^{-1});.$$ Comme $V$ est d'indice $c$
dans $U$, on a $\d{Card } UgU/U \geq \frac{1}{c} \,\d{Card } V/(V\cap
gVg^{-1})$, ce qui prouve la premi\`ere in\'egalit\'e.
D'autre part, on a :$$\d{Card }UgU/U  =  \d{Card } U/(U\cap
gUg^{-1}) \leq  \d{Card } U/(V\cap gVg^{-1})\;.$$ A nouveau, par
d\'efinition de $c$, on a $\d{Card }UgU/U \leq c\, \d{Card }
V/(V\cap gVg^{-1})$.
\end{proof}

Notons $\Omega$ l'ensemble des places non-archim\'ediennes de $K$, fixons une
place $\omega\in \Omega$ et raisonnons dans le groupe $\bG(K_\omega)$. Le corps
$K_\omega$ est une extension finie de $\dr Q_p$ pour un certain nombre premier
$p$ ; on
note $q$ le cardinal de son corps r\'esiduel. Pour les r\'esultats sur les
groupes sur les corps locaux, nous nous r\'ef\'erons \`a
l'article de J. Tits (\cite{Tits}), dont nous reprenons les notations.

R\'esumons les objets fournis par la th\'eorie des groupes $p$-adiques dont
nous aurons besoin : on dispose dans $\bG(K_\omega)$ d'un tore maximal
$K_\omega$-d\'eploy\'e $T_\omega$, et on note $N_\omega$ son normalisateur et
$Z_\omega$ son
centralisateur (ce sont des $K_\omega$-sous-groupes de $\bG$).

De plus on d\'efinit les objets suivants :

\begin{enumerate}
\item Les groupes $X^*=\d{Hom}_{K_\omega}(T_\omega,\d{Mult})$ et
$X_*=\d{Hom}_{K_\omega}(\d{Mult},
T_\omega)$ des caract\`eres et co-caract\`eres
d\'efinis
sur $K_\omega$ du tore, ainsi que
$X^*(Z)=\d{Hom}_{K_\omega}(Z_\omega,\d{Mult})$.
\item L'espace vectoriel $V=\dr R \otimes X_*$, et le syst\`eme de racines
restreintes $\Phi \subset X^*$ associ\'e au tore $T_\omega$.
\item Une application $\nu$ de $N_\omega(K_\omega)$ dans le groupe des
transformations
affines d'un espace $A$ sous $V$, d\'efinie en 1.2 de \cite{Tits} comme l'unique
extension de l'application de $Z_\omega(K_\omega)$ v\'erifiant (en notant
$v_\omega$ la valuation associ\'ee \`a $\omega$) : $$\d{pour tous} \; z \in
Z_\omega(K_\omega)\textrm{ et }\chi \in X^*(Z)\textrm{, on
a }\chi(\nu(z))=v_\omega(\chi(z))\; .$$
\item Le groupe de Weyl fini ${}^v W = N_\omega(K_\omega)/Z_\omega(K_\omega)$
et le groupe $\tilde{W}= N_\omega(K_\omega)/\d{ker}(\nu)$ qui contient le groupe
de Weyl affine $W$ comme sous-groupe distingu\'e d'indice fini. On identifie
$\tilde W$ comme un sous-groupe des transformations affines de $V$ en
choisissant dans $A$ un point sp\'ecial comme origine. ${}^v W$ est alors
l'ensemble des automorphismes de $\tilde W$ fixant l'origine.
\item Un choix d'un ensemble $\Phi^+$ de racines positives dans $\Phi$, et donc
une chambre $C$ contenant 0 dans $V$ d\'efinie comme l'ensemble des points $v$
de
$V$ tels que pour tout $\chi\in \Phi^+$, on a $\chi (v)\geq 0$.
\item La chambre vectorielle $Y^+ =\dr R^+ \otimes C$ de $V$ et un sous-groupe
compact ouvert $U_\omega$ de $\bG(K_\omega)$ (le fixateur du point sp\'ecial)
tels que
$\bG(K_\omega)$ est l'union des doubles classes $U_\omega a U_\omega$ pour $a
\in
Z_\omega^+=\nu^{-1}(Y^+)$ (d\'ecompostion de Cartan)
\end{enumerate}

De plus, si $n \in Z_\omega$ est tel que $\nu(n)$ est dans ${}^v W$, alors $n$
appartient
\`a $U_\omega$.

Enfin, on peut d\'efinir sur $\tilde W$ une fonction longueur pond\'er\'ee \`a
valeur enti\`ere (voir le paragraphe 3.3 de \cite{Tits}) de la fa\c con suivante
: on note $(r_i)$ les sym\'etries de $W$ associ\'ees \`a un
syst\`eme de racines simples dans $\Phi^+$. A chacun de ces \'el\'ements est
associ\'e un entier non nul $d(r_i)$. On \'ecrit tout \'el\'ement $w \in \tilde
W$ sous la
forme $w=r_{i_1} \ldots r_{i_l} w_0$ o\`u $w_0(C)=C$ et $r_{i_1} \ldots r_{i_l}$
est un mot r\'eduit dans $W$. On pose alors $l(w)=d(r_{i_1}) + \ldots
+d(r_{i_l})$.

Alors, d'apr\`es la section 3.3 de \cite{Tits} (voir aussi \cite{Gross}, 7.3),
pour tout $a \in Z_\omega^+$, en notant $\nu(a)=w$, on a : $$
\d{Card}\left(U_\omega aU_\omega/U_\omega\right) = \frac{\sum_{y \in {}^v W w
{}^vW}
q^{l(y)}}{\sum_{y \in {}^vW} q^{l(y)}}$$

On tire le r\'esultat suivant de cette formule :

\begin{lemm}\label{vol}
\begin{enumerate}
\item si $a \in Z_\omega^+$ n'est pas dans $U_\omega$, on a :
$$\d{Card}\left(U_\omega aU_\omega/U_\omega\right)
\geq p\,.$$
\item si $a_n$ tend vers l'infini dans $Z_\omega^+$, on a
$\d{Card}\left(U_\omega a_nU_\omega/U_\omega\right) \to + \infty$.
\end{enumerate}
\end{lemm}

\begin{proof}
Commen\c cons par le second point : si $a_n$ tend vers l'infini dans $\dr
Z_\omega^+$, alors la longueur $l(\nu(a))$ aussi, ce qui suffit.

Pour le premier point : soit $a$ un \'el\'ement de $Z_\omega^+$ qui n'est pas
dans
$U_\omega$. Consid\'erons $w_0$ le mot le plus court dans ${}^vW \nu(a)$.

Si $w_0$ ne
fixe pas $C$, alors $l(w_0) \geq 1$ et pour tout $w\in {}^vW$, on a par
d\'efinition $l(ww_0)=l(w)+l(w_0)$. On en d\'eduit que
$\d{Card}\left(U_\omega aU_\omega/U_\omega\right)$ est plus grand que
$q^{l(w_0)}$, donc que
$p$.

Si $w_0$ fixe $C$, alors $w_0$ n'est pas dans ${}^v W$ (sinon $a$
appartient \`a $U_\omega$). $w_0^{-1}$ envoie donc l'origine de $V$ sur un
autre point $x$. Or il existe dans ${}^v W$ une sym\'etrie $s$ qui envoie $x$
sur un point n'appartenant pas \`a $C$. Alors, le point $w_0 s w_0^{-1}$ est
dans $W$ car $W$ est distingu\'e dans $\tilde W$, mais pas dans ${}^v W$, car
l'origine est envoy\'ee sur un point en dehors de $C$, donc n'est pas fix\'ee.

Donc $w_0s$ s'\'ecrit sous la forme $r w_0$, o\`u $r$ est dans $W$ mais pas
dans ${}^v W$. En raisonnant comme pr\'ecedemment, mais pour l'ensemble ${}^v
Wrw_0$, on obtient aussi dans ce cas que
$\d{Card}\left(U_\omega aU_\omega/U_\omega\right)$ est
plus
grand que $p$.
\end{proof}

Maintenant que nous disposons de tous ces objets, le lemme \ref{lem:volume}
peut \^etre prouv\'e :

\begin{proof}[Preuve du lemme \ref{lem:volume}]
On fixe maintenant le compact ouvert $U_0= \prod_{\omega \in
\Omega} U_\omega$.
Consid\'erons une suite $(g_n)$ d'\'el\'ement de $G^f$ qui
tend vers l'infini. Chaque $g_n$ s'\'ecrit comme une suite
$(g_{\omega,n})_{\omega\in\Omega}$ dans le produit $
\prod_{\omega\in\Omega} \bG(K_\omega)$.
On d\'ecompose toutes les coordon\'ees dans la d\'ecomposition de Cartan :
$g_{\omega,n} \in U_\omega a_{\omega,n} U_\omega\textrm{, avec }a_{\omega,n}\in
Z^+_\omega$.

Soit $A$ un entier positif. Alors on veut prouver qu'il existe $N$
tel que pour tout $n\geq N$, $\lambda(U_0 g_n U_0)\geq A$.
Pour tout \'el\'ement $g_n$ tel qu'il
existe une place $\omega_p$ au dessus d'un nombre premier $p \geq A$ avec
$a_{\omega_p,n}\not\in U_{\omega_p}$, on a d'apr\`es le premier
point du lemme \ref{vol} :
\begin{flushleft}$\lambda(U_0 g_n U_0)=\prod_{\omega\in\Omega}
\d{Card}\left(U_\omega a_{\omega,n}U_\omega/U_\omega\right) \geq$\\
\end{flushleft}
\begin{flushright}
$\d{Card}\left(U_{\omega_p}a_{\omega_p,n}U_{\omega_p}/U_{\omega_p}\right)\geq
p \geq A
\;.$\end{flushright}

Par ailleurs, consid\'erons la sous-suite $(g_n)_{n\in S}$ telle que pour tout
$n\in S$ et pour toute place $\omega$ au dessus d'un nombre premier $p\geq A$,
l'\'el\'ement $a_{\omega,n}$ appartient \`a $U_\omega$.
Si cette sous-suite est finie, le r\'esultat voulu est prouv\'e. Sinon,
comme $(g_n)$ sort de tout compact, il existe une place $\omega_p$ au dessus
 d'un nombre premier $p\leq A$ tel
que la suite $(a_{\omega_p,n})_{n\in S}$ tend vers l'infini dans
$Z^+_{\omega_p}$. Alors, d'apr\`es le deuxi\`eme point du lemme \ref{vol},
on a :
\begin{flushleft}$\lambda(U_0 g_n U_0)=\prod_{\omega\in\Omega}
\d{Card}\left(U_\omega a_{\omega,n}U_\omega/U_\omega\right)
\geq$\\ \end{flushleft} \begin{flushright}
$\d{Card}\left(U_{\omega_p}a_{{\omega_p},n}U_{\omega_p}/U_{\omega_p}\right)
\xrightarrow{n\to \infty} +\infty\;.$\end{flushright}

Ainsi pour $n$ suffisamment grand, le volume $\lambda(U_0 g_n U_0)$ est de toute
fa\c
con sup\'erieur \`a $A$.
Cela prouve le r\'esultat pour le groupe compact ouvert $U_0$. Or on a vu avec
le lemme \ref{lem:changer} que cela suffisait.
\end{proof}

\bibliographystyle{cdraifplain}
\bibliography{biblio}

\end{document}